\documentclass[a4paper,10pt,twoside]{article}

\def\ifundefined#1{\expandafter\ifx\csname#1\endcsname\relax}

\usepackage{graphicx}
\usepackage{graphics}
\usepackage{mathrsfs}
\usepackage{float}
\usepackage{rotating}

\usepackage[T1]{fontenc}
\usepackage[english]{babel}
\usepackage[utf8]{inputenc}
\usepackage{amsthm,amsmath,amssymb,upref,amscd,amsfonts}  
\usepackage{euscript}
\usepackage{epsfig}
\usepackage{ae}
\usepackage{aecompl}
\usepackage[metapost]{mfpic}
\usepackage{calc}
\usepackage{enumerate}
\usepackage{enumitem}
\usepackage{url}
\usepackage[affil-it]{authblk}
\usepackage{microtype}
\usepackage{cite}

\theoremstyle{plain}
\newtheorem{theorem}{Theorem}[section]
\newtheorem{proposition}[theorem]{Proposition}
\newtheorem{lemma}[theorem]{Lemma}
\newtheorem{corollary}[theorem]{Corollary}
\newtheorem{theorem*}{Theorem}

\theoremstyle{definition}

\newlength{\normalparindent}
\numberwithin{equation}{section}

\mathchardef\sa="303A
\renewcommand{\epsilon}{\varepsilon}

\newcommand{\R}{\ensuremath{\mathbf{R}}}

\newcommand{\ip}[3]{\ensuremath{( {#1}, \, {#2} )}}

\newcommand{\lm}[1][m]{\ensuremath{\lambda_{#1}}}
\newcommand{\um}[1][m]{\ensuremath{\mu_{#1}}}
\newcommand{\tk}[1][k]{\ensuremath{\tau_{#1}}}
\newcommand{\Jm}[1][m]{\ensuremath{J_{#1}}}

\newcommand{\kk}[1][k]{\ensuremath{\kappa_{#1}}}

\newcommand{\Llm}[1][m]{\ensuremath{\Lambda_{#1}(\Oe)}}
\newcommand{\Lum}[1][m]{\ensuremath{\Lambda_{#1}(\Ot)}}

\newcommand{\vp}{\ensuremath{\varphi}}
\newcommand{\vs}{\ensuremath{\psi}}

\newcommand{\Psh}{\ensuremath{\Psi}}
\newcommand{\Ph}{\ensuremath{\Phi}}
\newcommand{\Phz}{\ensuremath{\Phi_p}}
\newcommand{\Phr}{\ensuremath{\Phi_h}}

\newcommand{\B}{\ensuremath{B}}
\renewcommand{\S}{\ensuremath{S}}
\newcommand{\Se}{\ensuremath{S_1}}
\newcommand{\St}{\ensuremath{S_2}}
\newcommand{\Sj}{\ensuremath{S_j}}

\newcommand{\Ke}{\ensuremath{K_1}}
\newcommand{\Kt}{\ensuremath{K_2}}
\newcommand{\Kj}[1][j]{\ensuremath{K_{#1}}}

\newcommand{\N}{\ensuremath{N}}

\renewcommand{\H}{\ensuremath{H}}
\newcommand{\He}{\ensuremath{H_1}}
\newcommand{\Ht}{\ensuremath{H_2}}
\newcommand{\Hj}[1][j]{\ensuremath{H_{j}}}

\newcommand{\Xj}[1][j]{\ensuremath{X_{#1}}}

\newcommand{\Xm}[1][m]{\ensuremath{\mathcal{X}_{#1}}}

\newcommand{\Oe}{\ensuremath{\Omega_1}}
\newcommand{\Ot}{\ensuremath{\Omega_2}}
\newcommand{\Oj}[1][j]{\ensuremath{\Omega_{#1}}}

\newcommand{\delt}{\ensuremath{\delta}}
\newcommand{\deltf}{\ensuremath{\kappa}}

\newcommand{\Om}{\ensuremath{\Omega}}

\newcommand{\hhe}{\ensuremath{\partial_{x_n} \Ke \Se \vp}}
\newcommand{\hht}{\ensuremath{\partial_{x_n} \Kt \St \vp}}
\newcommand{\hhj}{\ensuremath{\partial_{x_n} \Kj \Sj \vp}}
\newcommand{\Bk}{\ensuremath{B}}
\newcommand{\Bkh}{\ensuremath{\frac{1}{2}B}}

\title{Asymptotics of Hadamard Type for Eigenvalues of the Neumann Problem on $C^1$-domains for Elliptic Operators}
\author{Johan Thim\footnote{\tt johan.thim@liu.se}}
\affil{\small Department of Mathematics, University of Link\"{o}ping, Link\"{o}ping, Sweden}

\date{\today}

\begin{document}

\maketitle

\begin{abstract}
\noindent
This article investigates how the eigenvalues of the Neumann problem for an elliptic operator depend on the domain in the case when the domains involved are of class~$C^1$. We consider the Laplacian and use results developed previously for the corresponding Lipschitz case. In contrast with the Lipschitz case however, in the~$C^1$-case 
we derive an asymptotic formula for the eigenvalues when the domains are of class~$C^1$. Moreover, as an application we consider the
case of a~$C^1$-perturbation when the reference domain is of class~$C^{1,\alpha}$. 

\bigskip


\medskip


\end{abstract}


\section{Introduction}
The results presented in this article are based on an abstract framework  
for eigenvalues of the Neumann problem previously developed by Kozlov and Thim~\cite{Kozlov2014},
where we considered applications to Lipschitz- and $C^{1,\alpha}$-domains. However, the corresponding result
for~$C^1$-domains was omitted. In this study we present an asymptotic formula of Hadamard type for perturbations in the case when the
domains are of class~$C^1$. We also apply this theorem to the case when the reference domain is~$C^{1,\alpha}$, which
simplifies the involved expressions.

Partial differential equations
are typically not easily solvable when the domain is merely~$C^1$. Indeed, existence results for
Laplace's equation on a general~$C^1$-domain with~$L^p$-data on the boundary was only finally 
resolved by Fabes et al.~\cite{Fabes1978} in~1978.
This problem was difficult due to the fact that proving that the layer potentials define compact operators (so Fredholm theory
is applicable similar to the~$C^{1,\alpha}$-case) was rather technical. The results are based on estimates for the 
Cauchy integral on Lipschitz curves and we only obtain~$L^p$-estimates for the gradient. As a consequence, the problem
of eigenvalue dependence on a~$C^1$-domain becomes difficult. 

Hadamard~\cite{Hadamard1968,Mazya1998} studied a special type of perturbations of domains with smooth boundary in the early twentieth century,
where the perturbed domain~$\Omega_{\epsilon}$ is represented by~$x_{\nu} = h(x')$
where~$x' \in \partial \Omega_0$,~$x_{\nu}$ is the signed distance to the boundary 
($x_{\nu} < 0$ for~$x \in \Omega_0$), and~$h$ is a smooth
function bounded by a small parameter~$\epsilon$.
Hadamard considered the Dirichlet problem, 
but a formula of Hadamard-type for the first nonzero eigenvalue of the Neumann-Laplacian is given by
\[
\Lambda(\Omega_{\epsilon}) = \Lambda(\Omega_0) + 
	\int_{\partial \Omega_0} h \bigl( |\nabla \vp|^2 - \Lambda(\Omega_0) \vp^2  \bigr) dS 
	+ o(\epsilon),
\]
where $dS$ is the surface measure on~$\partial \Omega_0$ and~$\vp$ is an eigenfunction corresponding
to~$\Lambda(\Omega_0)$ such that~$\| \vp \|_{L^2(\Omega_0)} = 1$; 
compare with Grinfeld~\cite{Grinfeld2010}. 
In more general terms, eigenvalue dependence on domain perturbations is a classical and important 
problem going far back. Moreover, these problems are closely related to shape optimization; 
see, e.g., Henrot~\cite{Henrot2006}, and Soko\l{}owski and Zol\'esio~\cite{Sokolowski1992}, 
and references found therein.

Specifically, let~$\Oe$ and~$\Ot$ be domains in~$\R^n$,~$n \geq 2$,
and consider the spectral problems
\begin{equation}
\label{eq:maineq}
\left\{
\begin{aligned}
& -\Delta  u = \Llm[{}] u   & & \mbox{in } \Oe,\\
&\partial_{\nu}  u  =  0 & & \mbox{on } \partial \Oe
\end{aligned} \right.
\end{equation}
and
\begin{equation}
\label{eq:maineq2}
\left\{
\begin{aligned}
& -\Delta v = \Lum[{}] v  & & \mbox{in } \Ot,\\
&\partial_{\nu} v  =  0 & & \mbox{on } \partial \Ot,
\end{aligned} \right.
\end{equation}
where~$\partial_{\nu}$ is the normal derivative with respect to the outward normal. In the case of nonsmooth boundary, we consider
the corresponding weak formulations. The analogous Dirichlet problems have previously been 
considered~\cite{Kozlov2006,Kozlov2011,Kozlov2012,Kozlov2013}, however the Neumann problem requires a different approach as to
what one can use as a proximity quantity between the two domains and the operators involved.

We will require that the domains are close in the sense that the Hausdorff distance between the sets~$\Oe$ and~$\Ot$, i.e.,
\begin{equation}
\label{eq:d_def}
d = \max \{ \sup_{x \in \Oe} \inf_{y \in \Ot} |x-y|, \; \sup_{y \in \Ot} \inf_{x \in \Oe} |x-y| \},
\end{equation}
is small. 
If, e.g., the problem in~(\ref{eq:maineq}) has a discrete spectrum and the two domains~$\Oe$
and~$\Ot$ are close, then the problem in~(\ref{eq:maineq2}) has 
precisely~$\Jm[m]$ eigenvalues~$\Lum[k]$ close to~$\Llm$; see for instance Lemma~3.1 in~\cite{Kozlov2014}.
Here,~$\Jm[m]$ is the dimension of the eigenspace~$\Xj[m]$ corresponding to~$\Llm$.
The aim is to characterize the difference~$\Lum[k] - \Llm$ for~$k=1,2,\ldots,\Jm[m]$.

In a previous study~\cite{Kozlov2014}, we considered the cases when the domains are Lipschitz or~$C^{1,\alpha}$,
with~$0 < \alpha < 1$, as applications of an abstract framework. 
The main result is an asymptotic result for~$C^{1,\alpha}$-domains, where~$\Oe$ is a~$C^{1,\alpha}$-domain 
and~$\Ot$ is a Lipschitz perturbation of~$\Oe$
in the sense that the perturbed domain~$\Ot$ can be characterized by a function~$h$
defined on the boundary~$\partial \Oe$ such that every point~$(x',x_{\nu}) \in \partial \Ot$ is 
represented by~$x_{\nu} = h(x')$, where~$(x',0) \in \partial \Oe$ and~$x_{\nu}$ is the signed distance to~$\partial \Oe$
as defined above. 
Moreover, the function~$h$ is assumed to be Lipschitz continuous and satisfy~$|\nabla h| \leq C d^{\alpha}$. 
We proved that if the
problem in~{\rm(}\ref{eq:maineq}{\rm)} has a discrete spectrum and~$m$ is fixed, 
then there exists a constant~$d_0 > 0$ such that if~$d \leq d_0$, then 
\begin{equation}
\label{i:eq:holder_main_boundary}
\begin{aligned}
\Lum[k] - \Llm[m] = {} & 
\kk + O(d^{1+\alpha})
\end{aligned}
\end{equation}
for every~$k = 1,2,\ldots,\Jm[m]$.
Here~$\kk[{}] = \kk$ is an eigenvalue of the problem
\begin{equation}
\label{eq:T2_tk_intro}
\kk[{}] \ip{\vp}{\vs}{1} = 
\int_{\partial \Oe} h(x') \bigl( 
	\nabla \vp \cdot \nabla \vs 
	- \Llm \vp \vs
	\bigr) dS(x')
\quad \mbox{for all } \vs \in \Xj[m], 
\end{equation}
where~$\vp \in \Xj[m]$. Moreover, $\kk[1],\kk[2],\ldots,\kk[\Jm]$ in~{\rm(}\ref{i:eq:holder_main_boundary}{\rm)}
run through all eigenvalues of~{\rm(}\ref{eq:T2_tk_intro}{\rm)} counting their multiplicities; see Theorem~1.1 in~\cite{Kozlov2014}.

In the case when the domains are merely Lipschitz, we only obtain that 
there exists a constant~$C$, independent of~$d$, such that~$|\Lum[k]- \Llm[m]| \leq C d $
for every~$k=1,2,\ldots,\Jm$; see Corollary 6.11 in~\cite{Kozlov2014}.
Furthermore, in Section~6.7 of~\cite{Kozlov2014} we provide an example which shows that we can not get an asymptotic 
result of the type above for the Lipschitz case.

\subsection{New Results}
The main result of this article is proved in Section~\ref{s:mainresults}, where an asymptotic formula for~$\Lum - \Llm[k]$ in the case of~$C^1$-domains
is derived. The main term consists of extensions of eigenfunctions to~(\ref{eq:maineq}) and the remainder is of
order~$o(d)$; see Theorem~\ref{t:c1}. 
We suppose that~$\Ot$ is a Lipschitz perturbation of a~$C^1$-domain~$\Oe$ such that the 
Hausdorff distance~$d$ between~$\Oe$ and~$\Ot$ is small
and the outward normals~$n_1$ and~$n_2$ --- taken at the corresponding points of~$\Oe$ and~$\Ot$, respectively --- 
are comparable in the sense that~$n_1 - n_2 = o(1)$ as~$d \rightarrow 0$ (uniformly). 
If we also require that~$\Ot \subset \Oe$ to avoid the need for extension theorems, we obtain the following result. 

\begin{theorem}
\label{t:i:c1}
Suppose that~$\Oe$ is a~$C^1$-domain, that~$\Ot$ is as described above, and that~$\Ot \subset \Oe$.
In addition, assume that the problem in~{\rm(}\ref{eq:maineq}{\rm)} has a discrete spectrum and that~$m$ is fixed. 
Then there exists a constant~$d_0 > 0$ such that if~$d \leq d_0$, then 
\begin{equation}
\label{eq:i:c1_main}
\Lum[k] - \Llm  = \tk + o(d) 
\quad \mbox{for } k = 1,2,\ldots,\Jm. 
\end{equation}
Here,~$\tau = \tk$ is an eigenvalue of
\begin{equation}
\label{eq:i:c1_main_eig}
\tau \ip{\vp}{\vs}{1} = 
 \int_{\Oe \setminus \Ot} \bigl( 
	\nabla \vp \cdot \nabla \vs 
	 - \Llm \vp  \vs
	\bigr) \, dx
\quad \mbox{for all } \vs \in \Xj[m], 
\end{equation}
where~$\vp \in \Xj[m]$. 
Moreover,~$\tk[1], \tk[2], \ldots, \tk[\Jm]$ in~{\rm(}\ref{eq:i:c1_main}{\rm)}
run through all eigenvalues of~{\rm(}\ref{eq:i:c1_main_eig}{\rm)} counting their multiplicities.
\end{theorem}

\noindent Note that the main term is of order~$d$ and that the remainder is strictly
smaller as~$d \rightarrow 0$.

As an application, we then in Section~\ref{s:app} consider the case when the perturbation is of Hadamard type
and we assume that the reference domain~$\Oe$ is a~$C^{1,\alpha}$-domain. 
Indeed, if~$\Ot$ is a perturbation of~$\Oe$
in the sense that the perturbed domain~$\Ot$ can be characterized by a Lipschitz function~$h$
defined on the boundary~$\partial \Oe$ such that~$(x',x_{\nu}) \in \partial \Ot$ is 
represented by~$x_{\nu} = h(x')$, where~$(x',0) \in \partial \Oe$,~$x_{\nu}$ is 
the signed distance to~$\partial \Oe$ as defined above, and~$\nabla h = o(1)$ as~$d \rightarrow 0$ (uniformly),
we obtain the following result; see Theorem~\ref{t:c1_boundary}. 

\begin{theorem}
\label{i:t:c1_boundary}
Suppose that~$\Oe$ is a~$C^{1,\alpha}$-domain, that~$\Ot$ is a perturbation as described above, that the
problem in~{\rm(}\ref{eq:maineq}{\rm)} has a discrete spectrum, and that~$m$ is fixed. 
Then, there exists a constant~$d_0 > 0$ such that if~$d \leq d_0$, then 
\begin{equation}
\label{i:eq:c1_main_boundary}
\begin{aligned}
\Lum[k] - \Llm[m] = {} & 
\kk + o(d)
\end{aligned}
\end{equation}
for every~$k = 1,2,\ldots,\Jm[m]$. 
Here~$\kk[{}] = \kk$ is an eigenvalue of the problem
\begin{equation}
\label{eq:c1_T2_tk_intro}
\kk[{}] \ip{\vp}{\vs}{1} = 
\int_{\partial \Oe} h(x') \bigl( 
	\nabla \vp \cdot \nabla \vs 
	- \Llm \vp \vs
	\bigr) dS(x')
\quad \mbox{for all } \vs \in \Xj[m], 
\end{equation}
where~$\vp \in \Xj[m]$. 
Moreover, $\kk[1],\kk[2],\ldots,\kk[\Jm]$ in~{\rm(}\ref{i:eq:c1_main_boundary}{\rm)}
run through all eigenvalues of~{\rm(}\ref{eq:c1_T2_tk_intro}{\rm)} counting their multiplicities.
\end{theorem}

\noindent
We also note here that Theorem~\ref{i:t:c1_boundary} is sharp. Indeed, the main term in~(\ref{eq:c1_T2_tk_intro}) is of order~$d$ and 
the example given in Section~6.7 in~\cite{Kozlov2014} shows that this can not be improved.

\section{Notation and Definitions}
We will use the same abstract setting and notation that was used in Kozlov and Thim~\cite{Kozlov2014}.
Let us summarize the notation.
We consider the operator~$1 -\Delta$ and a number~$\lm[{}]$ is an eigenvalue
of the operator~$1 - \Delta$ if and only if~$\lm[{}] - 1$ is an eigenvalue of~$-\Delta$. The reason for considering~$1 - \Delta$
is to avoid technical difficulties due to the eigenvalue zero.
Enumerate the eigenvalues~$\Llm[k] = \lm[k]-1$,~$k=1,2,\ldots$, of~(\ref{eq:maineq}) 
according to $0 < \lm[1] < \lm[2] < \cdots$.
Similarly, we let~$\Lum[k] = \um[{}] - 1$ be the eigenvalues of~(\ref{eq:maineq2}).
Suppose that~$\He$ and~$\Ht$ are infinite dimensional subspaces of a Hilbert space~$H$.
We denote the inner product on~$H$ by~$( \, \cdot \, , \, \cdot \,)$.
Let the operators~$\Kj \colon \Hj \rightarrow \Hj$ be
positive definite and self-adjoint for~$j=1,2$. Furthermore, let~$\Ke$ be compact.
We consider the spectral problems
\begin{equation}
\label{eq:eig_k1}
\Ke \vp = \lambda^{-1} \vp, \quad \vp \in \He,
\end{equation}
and
\begin{equation}
\label{eq:eig_k2}
\Kt U = \mu^{-1} U, \quad U \in \Ht,
\end{equation}
and denote by~$\lm[k]^{-1}$ for~$k=1,2,\ldots$ the eigenvalues of~$\Ke$.
Let~$\Xj[k] \subset \He$ be the eigenspace corresponding to eigenvalue~$\lm[k]^{-1}$. Moreover, we denote
the dimension of~$\Xj[k]$ by~$\Jm[k]$ and define~$\Xm = \Xj[1] + \Xj[2] + \cdots \Xj[m]$,
where~$m \geq 1$ is any integer. 
In this article we study eigenvalues of~(\ref{eq:eig_k2}) located in a neighborhood of~$\lm^{-1}$,
where~$m$ is fixed.
Note that it is known that there are precisely~$\Jm[m]$ eigenvalues of~(\ref{eq:maineq2}) near~$\lm^{-1}$; 
see, e.g., Lemma~3.1 in~\cite{Kozlov2014}.
We wish to describe how close they are in the case of $C^1$-domains.

Let~$\Se \colon \H \rightarrow \He$ and~$\St \colon \H \rightarrow \Ht$ be orthogonal projectors and
define~$\S$ as the restriction of~$\St$ to~$\He$. 
To compare~$\Ke$ and~$\Kt$, we define the operator~$B \colon \He \rightarrow \Ht$ as
$\B = \Kt \S - \S \Ke$. For~$\vp \in \Xm$,~$\B \vp$ is typically small in applications. 
Furthermore, we use the convention that~$C$ is a generic constant that can change from line
to line, but always depend only on the parameters. We also use the notation~$\deltf$ for
a generic function~$\deltf \colon [0,\infty) \mapsto [0,\infty)$ such that~$\deltf(\delt) = o(1)$ as~$\delt \rightarrow 0$.


\subsection{Domains in~$\R^n$}
\label{s:domains}
Let~$\Oe$ be the reference domain which will be fixed throughout. 
We will assume that~$\Oe$ and~$\Ot$ are at least Lipschitz domains.
Then there exists a positive constant~$M$ such that the boundary~$\partial \Oe$ can be covered by
a finite number of balls~$B_k$,~$k=1,2,\ldots,N$, where there exists orthogonal coordinate systems in which
\[
\Omega_1 \cap B_k= \{ y = (y', y_n) \sa y_n > h_k^{(1)}(y')\} \cap B_k
\]
where the center of~$B_k$ is at the origin and~$h_k^{(1)}$ 
are Lipschitz functions, i.e., 
\[
|h_k^{(1)}(y') - h_k^{(1)}(x')| \leq M |y' - x'|,
\]
such that~$h_k^{(1)}(0) = 0$.
We assume that~$\Ot$ belongs to the class of domains where~$\Ot$ is close to~$\Oe$ in the sense that~$\Ot$ can be described by
\[
\Omega_2 \cap B_k = \{ y = (y', y_n) \sa y_n > h_k^{(2)}(y')\} \cap B_k,
\]
where~$h_k^{(2)}$ are also Lipschitz continuous with Lipschitz constant~$M$.

The case when~$\Oe$ is a~$C^1$- or $C^{1,\alpha}$-domain is defined analogously, with the addition that
that~$h_k^{(1)} \in C^{1}(\R^{n-1})$ (or~$C^{1,\alpha}(\R^{n-1})$) such that
\[
h_k^{(1)}(0) = \partial_{x_i} h_k^{(1)}(0) = 0, \quad i=1,2,\ldots,n-1.
\] 
Note that when~$\Oe$ is a~$C^1$-domain, we obtain that for~$P, Q \in \partial \Oe$, the outward normal~$n_1$ of~$\Oe$ satisfies
\[
n_1(P) - n_1(Q) = o(1) \quad \mbox{as } |P-Q| \rightarrow 0,
\]
uniformly. 

\subsection{Perturbations of~$C^{1}$-Domains}
\label{s:c1domains}
The situation we consider is the case when the reference domain~$\Oe$ is a~$C^1$-domain and the perturbed
domain~$\Ot$ is close in the sense of Section~\ref{s:domains}. We require that~$\Ot$ is a Lipschitz domain such
that
\begin{equation}
\label{eq:grad_hk}
| \nabla ( h_k^{(1)} - h_k^{(2)}) | = o(1), \quad \mbox{as } d \rightarrow 0, 
\end{equation}
uniformly. This condition can be compared to the one we used in~\cite{Kozlov2014} for
perturbations of~$C^{1,\alpha}$-domains: 
\begin{equation}
\label{eq:grad_hk_c1a}
| \nabla ( h_k^{(1)} - h_k^{(2)}) | \leq C d^{\alpha}.
\end{equation}
Note that~$h_k^{(2)}$ are only assumed to be Lipschitz continuous and satisfy~(\ref{eq:grad_hk}) and~(\ref{eq:grad_hk_c1a}), respectively.

\section{Definition of the Operators~$\Kj$}
Let~$\Oe$ and~$\Ot$ be two domains in~$\R^n$ ($\Oe \cap \Ot \neq \emptyset$)
and put~$\H = L^2(\R^n)$ and~$\Hj = L^2(\Oj)$ for~$j = 1,2$, where
functions in~$\Hj$ are extended to~$\R^n$ by zero outside of~$\Oj$ in necessary.
For~$f \in L^2(\Oj)$,
the weak solution to the Neumann problem~$(1 - \Delta) W_j = f$ in~$\Oj$ and~$\partial_{\nu} W_j = 0$ on~$\partial \Oj$
for~$j=1,2$ satisfies 
\[
\int_{\Oj} ( \nabla W_j \cdot \nabla v  + W_j v ) \, dx = \int_{\Oj} f  v \, dx
\quad \mbox{for every } v \in H^1(\Oj),
\]
and the Cauchy-Schwarz inequality implies that
\[
\| \nabla W_j \|_{L^2(\Oj)} + \| W_j \|_{L^2(\Oj)} \leq \| f \|_{L^2(\Oj)}
\quad \mbox{for all } f \in L^2(\Oj).
\]
We define the operators~$\Kj$ on~$L^2(\Oj)$,~$j=1,2$, as the solution operators corresponding to the domains~$\Oj$,
i.e.,~$\Kj f = W_j$.
The operators~$\Kj$ are self-adjoint and positive definite, and if~$\Oj$ are, e.g., Lipschitz, also compact.

\subsection{Results for Lipschitz Domains}
\label{s:def_lipdom}
We will work with results for Lipschitz domains and then refine estimates using the additional smoothness
of the~$C^1$-case.
Let~$\Om$ be a Lipschitz domain.
The truncated cones~$\Gamma(x')$ at~$x' \in \partial \Om$ are given by, e.g.,
\[
\Gamma(x') = \{ x \in \Om \sa  |x - x'| < 2 \mbox{dist}(x, \partial \Om) \}
\]
and the non-tangential maximal function is defined on the boundary~$\partial \Om$ by
\[
\N(u)(x') = \max_{k=1,2,\ldots,N} \sup \{ |u(x)| \sa x \in \Gamma(x') \cap B_k \} .
\]
For the case when~$\Oe$ and~$\Ot$ are Lipschitz, one can show that 
\begin{equation}
\label{eq:est_max_Kj}
\| \N(\Kj u) \|_{L^2(\partial \Oj)} + \| \N(\nabla \Kj u) \|_{L^2(\partial \Oj)} \leq C \| u \|_{L^2(\Oj)}, \quad j=1,2,
\end{equation}
where the constant~$C$ depends only on the Lipschitz constant~$M$ and~$B_1,B_2,\ldots,B_N$.
We interpret~$\partial_{\nu} \Kj u = 0$ on~$\partial \Oj$ in the sense 
that~$n \cdot \nabla \Kj u \rightarrow 0$
nontangentially (limits taken inside cones~$\Gamma(x')$) at almost every point on~$\partial \Om$, where~$n$ is the outward normal. 
These results are discussed further in Section~6.2 of Kozlov and Thim~\cite{Kozlov2014}.
Let us summarize Lemmas~6.2 and~6.3 in~\cite{Kozlov2014} for convenience.

\begin{lemma}
\label{l:lap_neu}
Let~$\Om$ be a Lipschitz domain. Then,
\begin{enumerate}
\item[{\rm(i)}] if~$g \in L^2(\partial \Om)$, then there exists a unique {\rm(}up to constants{\rm)} function~$u$ in~$H^1(\Om)$ 
such that~$(1 - \Delta)u = 0$ in~$\Om$ and~$\partial_{\nu} u = g$ a.e.\ on~$\partial \Om$ in the nontangential sense, 
and moreover,
\[
\| \N(u) \|_{L^2(\partial \Om)} + \| \N(\nabla u) \|_{L^2(\partial \Om)} \leq C \| g \|_{L^2(\partial \Om)};
\]
\item[{\rm(ii)}]
if~$f \in L^2(\Om)$, then there exists a unique function~$u$ in~$H^1(\Om)$
such that~$(1 - \Delta)u = f$ in~$\Om$, and~$\partial_{\nu} u = 0$ on~$\partial \Om$ in the nontangential sense, and
\begin{equation*}
\| \N(u) \|_{L^2(\partial \Om)} + \| \N(\nabla u) \|_{L^2(\partial \Om)} \leq C \| f \|_{L^2(\Om)}.
\end{equation*}
\end{enumerate}
Here, the constant~$C$ depends only on~$M$ and~$B_1,B_2,\ldots,B_N$
\end{lemma}

\noindent
The corresponding lemma for the Dirichlet case is also known, and one can prove it using an argument 
similar to the one used to prove Lemmas~6.2 and~6.3 in~\cite{Kozlov2014}.

\begin{lemma}
\label{l:lap_dir}
Let~$\Om$ be a Lipschitz domain. Then,
\begin{enumerate}
\item[{\rm(i)}]
if~$g \in L^2(\partial \Om)$, then there exists a unique function~$u \in H^1(\Om)$
such that~$(1 - \Delta)u = 0$ in~$\Om$,~$u = g$ on~$\partial \Om$ in the nontangential sense,
and
\[
\| \N(u) \|_{L^2(\partial \Om)} \leq C \| g \|_{L^2(\partial \Om)};
\]
\item[{\rm(ii)}]
if~$f \in L^2(\Om)$, then there exists a unique function~$u \in H^1(\Om)$ such that~$(1-\Delta)u = f$
in~$\Om$,~$u = 0$ on~$\partial \Om$ in the nontangential sense,  
and
\[
\| \N(u) \|_{L^2(\partial \Om)} \leq C \| f \|_{L^2(\Om)}.
\]
\end{enumerate}
Here, the constant~$C$ depends only on~$M$ and~$B_1,B_2,\ldots,B_N$.
\end{lemma}

\noindent
We conclude with an extension result for Lipschitz domains; see, e.g.,~\cite[Lemma~6.4(i)]{Kozlov2014} for a proof.

\begin{lemma}
\label{l:ext}
Suppose that~$f \in H^1(\partial \Om)$ and~$g \in L^2(\partial \Om)$,
where~$\Om$ is a Lipschitz domain. Then 
there exists a function~$u \in H^1(\Om^c)$ such that~$u \rightarrow f$
and~$n \cdot \nabla u \rightarrow g$ nontangentially at almost every point on~$\partial \Om$,
where~$n$ is the outward normal of~$\Om$,
and there exists a constant~$C$ such that
\[
\| \N(u) \|_{L^2(\partial \Om)} + \|\N(\nabla u)\|_{L^2(\partial \Om)}
\leq
C ( \| f \|_{H^1(\partial \Om)} + \|g\|_{L^2(\partial \Om)} ),
\]
where~$C$ depends on~$M$ and~$B_1,B_2,\ldots,B_N$.
\end{lemma}

\section{Main Results}

Let us proceed to prove the main results. In Section~\ref{s:c1}, we prove a key lemma concerning an estimate
for~$\partial_{\nu} \Kj \Sj \vp$ on~$\partial (\Oe \cap \Ot)$. Using this estimate, we can refine results for
Lipschitz domains that were previously developed in~\cite{Kozlov2014}, and as a result,
obtain an asymptotic formula describing the difference between~$\lm^{-1}$ and~$\um^{-1}$ in terms
of eigenfunctions of~$\Ke$.

\subsection{Boundary Estimates for $C^1$-domains}
\label{s:c1}

Since~$\partial_{\nu} \vp = 0$ on~$\partial \Oe$, we would expect that~$\partial_{\nu} \vp$ is small also on~$\Ot$
if the domains are close. 
However, since in the~$C^1$-case, we only obtain solutions with derivatives in~$L^p$, this problem becomes
more difficult than the corresponding issue in the $C^{1,\alpha}$-case (which was solved in~\cite{Kozlov2014}). 
To this end, we will exploit that locally on the boundaries~$\partial \Oj$, the normal vectors can be approximated
by constant unit vectors $e_n$ (with respect to the local coordinate system). That is, we approximate the
surface by its tangent plane at a specific point. We obtain the following result.

\begin{lemma}
\label{l:est_kj_bndry_local}
Let~$P \in \partial (\Oe \cap \Ot)$ and~$\delt > 0$ such that~$B(P, 2\delt) \subset B_k$ for some~$k$, where~$B_k$
are the balls covering~$\Oe \cap \Ot$ given in Section~\ref{s:domains}. 
Then, there exists a function~$\deltf(\delt)$ such that
\begin{equation}
\label{eq:c1_g}
\int_{\partial (\Oe \cap \Ot) \cap B(P,\delt)} |\partial_{\nu} \Kj \Sj \vp|^2 \, dS(x') \leq \deltf(\delta) \, 
\int_{\Oe} |\vp|^2 \, dx, 
\quad j=1,2,
\end{equation}
for every~$\vp \in \Xm$,  
where~$\deltf(\delt) = o(1)$ as~$\delt \rightarrow 0$.
\end{lemma}

\begin{proof}
Let~$\Bk = B(P, \, 2\delt)$.
We wish to consider~$\partial_{\nu} \Kj \Sj \vp$ on~$\partial(\Oe \cap \Ot)$. However, since~$\nabla \Kj \Sj \vp$ only exist
in the sense of~$L^2$, it is nontrivial to exploit the fact that~$\partial_{\nu} \Kj \Sj \vp$ is zero on~$\partial \Oj$.
Therefore, let us instead consider~$\hhj$ (with respect to the coordinate system in~$B_k$). 
The outward normal of~$\Oj$ is
comparable to~$e_n$ in~$B_k$ and~$\partial_{\nu} \Kj \Sj \vp = 0$ on~$\partial \Oj$, 
so we expect~\hhj{} to be small on~$\partial \Oj \cap B_k$. 
Indeed, since~$\nabla \Kj \Sj \vp \cdot n_j \rightarrow 0$ nontangentially on~$\partial \Oj$ 
and~$n_j = e_n + o(1)$ as~$\delt \rightarrow 0$, we obtain that
\begin{equation}
\label{eq:est_hh_Oe}
\int_{\partial \Oj \cap \Bk} | \hhj{} |^2 \, dS(x') \leq \deltf(\delt) \int_{\Oe} |\vp|^2 \, dx . 
\end{equation}

However, we can not expect~$\hhj{}$ to be small on all of~$\Oj$.
The idea is to use the fact that~$\partial_{x_n}$ commutes with~$(1 - \lm -\Delta)$.
Indeed, we see that if~$\Ph = \hhe$,
then~$(1 -\lm - \Delta) \Ph = 0$ in~$\Oe$ and~$\Ph = \hhe$ on~$\partial \Oe$. The case when~$j = 2$ will be treated similarly
but requires some additional steps.
Let us consider the equation $(1 -\lm - \Delta) \Ph = 0$ in~$\Oe$ and~$\Ph = \hhe$ on~$\partial \Oe$.
We split this equation in two separate parts.

{\it Part 1}. Let~$\Phz$ be the solution to~$(1 -\lm - \Delta) \Phz = 0$ in~$\Oe$,~$\Phz = \hhe$ on~$\partial \Oe \cap \Bk$,
and on~$\partial \Oe \cap \Bk^c$, we let~$\Phz = 0$.
Lemma~\ref{l:lap_dir} implies that~$\Phz$ satisfies
\begin{equation}
\int_{\partial \Oe} |N(\Phz)|^2 \, dS(x') \leq \deltf(\delt) \int_{\Oe} |\vp|^2 \, dx. 
\end{equation}
Then it follows that
\begin{equation}
\label{eq:N_Phz_est}
\int_{\Oe \cap \partial \Ot \cap \Bk} |\Phz|^2 \, dS(x') \leq \deltf(\delt) \int_{\Oe} |\vp|^2. 
\end{equation}

{\it Part 2}. Let~$\Phr$ be the solution to~$(1 -\lm - \Delta) \Phr = 0$ in~$\Oe$,~$\Phr = 0$ 
on~$\partial \Oe \cap \Bk$, 
and~$\Phr = \hhe$ on~$\partial \Oe \cap \Bk^{c}$. 
To prove an estimate for~$\Phr$ on~$\partial \Oe \cap \Bk$ similar to the one given for~$\Phz$ in~(\ref{eq:N_Phz_est}), 
we use a local estimate for solutions to the Dirichlet problem where we exploit that
the boundary data is zero on~$\Oe \cap \Bk$. 
Indeed, let~$\Bkh$ be the ball with the same center as~$\Bk$ but half the radius. 
Then, e.g., Theorem~5.24 in Kenig and Pipher~\cite{Kenig1993}, implies that
\begin{equation}
\label{eq:aaam1}
\int_{\partial \Oe \cap \Bkh} | N(\nabla \Phr)|^2 \, dS(x') \leq C \int_{\Oe \cap \Bk} |\nabla \Phr |^2 \, dx
\end{equation}
since the tangential gradient of~$\Phr$ is zero on the boundary. This, in turn, implies that the left-hand side in~(\ref{eq:aaam1})
is finite, and furthermore, since also~$\Phr = 0$ on~$\Oe \cap \Bk$, it follows that
\begin{equation}
\label{eq:N_Phr_est}
\int_{\Oe \cap \partial \Ot \cap \Bkh} | \Phr |^2 \, dS(x') \leq C d \int_{\Oe} | \vp |^2 \, dx,
\end{equation}
where~$d$ is the Hausdorff distance between~$\Oe$ and~$\Ot$. 

Equations~(\ref{eq:N_Phz_est}) and~(\ref{eq:N_Phr_est}) are sufficient to obtain that
\[
\int_{\partial (\Oe \cap \Ot) \cap \Bkh}  |N(\hhe)|^2 \, dS(x') \leq \deltf(\delt) \int_{\Oe} |\vp|^2 \, dx
\]
since~$\Ph = \Phz + \Phr$.

Turning our attention to when~$j = 2$, we see that~$(1 - \Delta) \Kt \St \vp = \St \vp$ and that this equation is
not homogeneous. Moreover, the right-hand side is not necessarily small. 
However, since~$\S \vp = \lm \Kt \S \vp - \lm \B \vp$ and~$\B \vp$ is small, we can consider
\begin{equation}
\label{eq:aaa0}
(1 - \lm - \Delta) \Kt \St \vp = - \lm \B \vp.
\end{equation}
Let~$\Psh$ be the weak solution to~$(1-\lm -\Delta) \Psh = -\lm \B \vp$ in~$\Ot$ and~$\Psh = 0$ on~$\partial \Ot$. 
Then,~$\| \Psh \|_{H^1(\Ot)} \leq C \| \B \vp \|_{L^2(\Ot)}$ and the trace of~$\Psh$ is defined on~$\partial \Om$.
Moreover, from Lemma~\ref{l:lap_dir} we obtain that
\begin{equation}
\label{eq:aaa1}
\| N(\Psh) \|_{L^2(\partial \Ot)} \leq C \| \B \vp \|_{L^2(\Ot)}.
\end{equation}
Now, put~$\Phi = \Psh + W$. Then~$(1 - \lm - \Delta) W = 0$ and~$W = \hht$ on~$\partial \Ot$. 
It is now possible to carry out steps 1 and 2 for~$W$ in~$\Ot$ analogously with~$\Ph$ in~$\Oe$, exchanging the
roles of~$\Oe$ and~$\Ot$.
Thus, using the same notation, we obtain that
\begin{equation}
\label{eq:aaa2}
\int_{\partial (\Oe \cap \Ot) \cap \Bkh}  |N(W)|^2 \, dS(x') \leq \deltf(\delt) \int_{\Oe} |\vp|^2 \, dx.
\end{equation}
Finally, Lemma~6.6 in~\cite{Kozlov2014} states that~$\| \B \vp \|^2_{L^2(\Ot)} \leq C d \|\vp\|^2_{L^2(\Oe)}$,
so this fact and equations~(\ref{eq:aaa1}) and~(\ref{eq:aaa2}) prove that 
\begin{equation}
\label{eq:aaa3}
\int_{\partial (\Oe \cap \Ot) \cap \Bkh}  |N(\hht)|^2 \, dS(x') \leq \deltf(\delt) \int_{\Oe} |\vp|^2 \, dx.
\end{equation}

We can now conclude the proof by observing that the outward normal on~$\partial (\Oe \cap \Ot)$ 
is given by~$n_1$ or~$n_2$ at almost every point,
and~$n_j = e_n + r_j$ with~$r_j = \deltf(\delt)$, $j=1,2$, so we obtain that  
\[
\int_{\partial (\Oe \cap \Ot) \cap \Bkh}  |\partial_{\nu} \Kj \Sj \vp|^2 \, dS(x') \leq \deltf(\delt) \int_{\Oe} |\vp|^2 \, dx. \qedhere
\]
\end{proof}

\noindent
The previous lemma is local in nature, but due to compactness we can prove the following corollary. 

\begin{corollary}
\label{c:est_kj_bndry}
There exists a constant~$d_0 > 0$ such that if~$d \leq d_0$, then
\begin{equation}
\label{eq:c1_g_global}
\int_{\partial (\Oe \cap \Ot)} |\partial_{\nu} \Kj \Sj \vp|^2 \, dS(x') \leq \deltf(d) \, 
\int_{\Oe} |\vp|^2 \, dx, 
\quad j=1,2,
\end{equation}
for every~$\vp \in \Xm$,  
where~$\deltf(d) = o(1)$ as~$d \rightarrow 0$.
\end{corollary}

\begin{proof}
By compactness, if~$d$ is small we can cover~$\partial(\Oe \cap \Ot)$ by a finite number of balls~$B(P, d)$
such that~$B(P, 2d) \subset B_k$ for some~$k$, where~$B_k$ are the covering balls from Section~\ref{s:domains}. 
By choosing~$d_0$ small enough and letting~$\delta = d$ in the previous Lemma, the result in the corollary now follows.
\end{proof}

\subsection{Proof of Theorem~\ref{t:i:c1}} 
\label{s:mainresults}
The following proposition is a reformulation of Proposition~6.10 in~\cite{Kozlov2014},
where the proof can also be found. The ``tilded'' expressions are the extensions of the corresponding
functions provided by Lemma~\ref{l:ext}. We will use this result and Corollary~\ref{c:est_kj_bndry} to prove Theorem~\ref{t:i:c1}.

\begin{proposition}
\label{p:lipschitz}
Suppose that~$\Oe$ and~$\Ot$ are Lipschitz domains in the sense of Section~\ref{s:domains}.
Then
\begin{equation}
\label{eq:lip_main}
\lm[m]^{-1} - \um[k]^{-1} = \tk + O(d^{3/2})
\quad \mbox{for } k = 1,2,\ldots,\Jm. 
\end{equation}
Here,~$\tau = \tk$ is an eigenvalue of
\begin{equation}
\label{eq:lip_main_eig}
\begin{aligned}
\tau \ip{\vp}{\vs}{1} = {} &
\lm^{-1} \int_{\Oe \setminus \Ot} \bigl( 
	(1 - \lm) \widetilde{\Kt \S \vp}  \vs
	+ \nabla \widetilde{\Kt \S \vp} \cdot \nabla \vs 
	\bigr) \, dx\\
& - 
\lm^{-1} \int_{\Ot \setminus \Oe} \bigl( 
	(1 - \lm) (\Kt \S \vp) \widetilde{\vs} 
	+ \nabla \Kt \S \vp \cdot \nabla \widetilde{\vs}  
	\bigr) \, dx 
\end{aligned}
\end{equation}
for all $\vs \in \Xj[m]$, where~$\vp \in \Xj[m]$. 
Moreover,~$\tk[1], \tk[2], \ldots, \tk[\Jm]$ in~{\rm(}\ref{eq:lip_main}{\rm)}
run through all eigenvalues of~{\rm(}\ref{eq:lip_main_eig}{\rm)} counting their multiplicities.
\end{proposition}

\noindent
Let us now prove a version of this proposition that holds specifically for~$C^1$-domains. 
We will show the following result.

\begin{theorem}
\label{t:c1}
Suppose that~$\Oe$ is a~$C^1$-domain and that~$\Ot$ is a perturbation in the sense of Section~\ref{s:c1domains} satisfying~{\rm(}\ref{eq:grad_hk}{\rm)}. 
Then
\begin{equation}
\label{eq:c1_main}
\lm[m]^{-1} - \um[k]^{-1} = \tk + o(d) 
\quad \mbox{for } k = 1,2,\ldots,\Jm. 
\end{equation}
Here,~$\tau = \tk$ is an eigenvalue of
\begin{equation}
\label{eq:c1_main_eig}
\begin{aligned}
\tau \ip{\vp}{\vs}{1} = {} &
\lm^{-1} \int_{\Oe \setminus \Ot} \bigl( 
	(1 - \lm) \vp  \vs
	+ \nabla \vp \cdot \nabla \vs 
	\bigr) \, dx\\
& - 
\lm^{-1} \int_{\Ot \setminus \Oe} \bigl( 
	(1 - \lm) \widetilde{\vp} \widetilde{\vs} 
	+ \nabla \widetilde{\vp} \cdot \nabla \widetilde{\vs}  
	\bigr) \, dx 
\end{aligned}
\end{equation}
for all $\vs \in \Xj[m]$, where~$\vp \in \Xj[m]$. 
Moreover,~$\tk[1], \tk[2], \ldots, \tk[\Jm]$ in~{\rm(}\ref{eq:c1_main}{\rm)}
run through all eigenvalues of~{\rm(}\ref{eq:c1_main_eig}{\rm)} counting their multiplicities.
\end{theorem}

\begin{proof}
We need to prove that~(\ref{eq:lip_main_eig}) can be expressed as~(\ref{eq:c1_main_eig}) up
to a term of order~$o(d)$.
Since~$\Kt \S \vp = \B\vp + \lm^{-1} \S \vp$, we let~$\widetilde{\Kt \S \vp} = \widetilde{\B\vp} + \lm^{-1} \widetilde{\vp}$,
where~$\widetilde{\B\vp}$ is the extension of~$\B \vp$ from~$\Oe \cap \Ot$, and~$\widetilde{\vp}$ is the extension of~$\vp$ from~$\Oe$, 
both provided by Lemma~\ref{l:ext}. We show that~$\widetilde{\B\vp}$ is small and that~$\lm^{-1} \widetilde{\vp}$ gives the main term.
To this end, let~$V = \B\vp$ in~$\Oe \cap \Ot$. 
Then~$(1 - \Delta)V = 0$ in~$\Oe \cap \Ot$,~$\partial_{\nu} V = \partial_{\nu} \Kt \S \vp$ on~$\partial \Oe \cap \Ot$,
and~$\partial_{\nu} V = -\partial_{\nu} \Ke \vp$ on~$\Oe \cap \partial \Ot$. Using Corollary~\ref{c:est_kj_bndry}
and Lemma~\ref{l:lap_neu}, we then obtain that
\[
\| N(V) \|_{L^2(\partial(\Oe \cap \Ot))} + \| N(\nabla V) \|_{L^2(\partial(\Oe \cap \Ot))} \leq \deltf(d) \| \vp \|_{L^2(\Oe)}, 
\]
where~$\deltf(d) = o(1)$ as~$d \rightarrow 0$,
and thus,
\[
\| N(\widetilde{\B\vp}) \|_{L^2(\partial(\Oe \cap \Ot))} + \| N(\nabla\widetilde{\B\vp}) \|_{L^2(\partial(\Oe \cap \Ot))} 
\leq \deltf(d) \| \vp \|^2_{L^2(\Oe)}. 
\]
Now, the Cauchy-Schwarz inequality implies that
\[
\begin{aligned}
\int_{\Oe \setminus \Ot} |\nabla \widetilde{\B \vp} \cdot \nabla \vs| \, dx  &\leq
\biggl( \int_{\Oe \setminus \Ot} |\nabla \widetilde{\B\vp}|^2 \, dx \biggr)^{1/2}
\biggl( \int_{\Oe \setminus \Ot} |\nabla \vs|^2 \, dx \biggr)^{1/2}\\
&\leq
C d \left( \int_{\partial(\Oe \cap \Ot)} N(\nabla \widetilde{\B \vp})^2 \, dS(x') \right)^{1/2}
\biggl( \int_{\Oe \setminus \Ot} |\nabla \vs|^2 \, dx \biggr)^{1/2}\\
& = o(d),
\end{aligned}
\]
and similarly,
\[
\begin{aligned}
\int_{\Oe \setminus \Ot} | \widetilde{\B \vp} \vs| \, dx  &\leq
C d \left( \int_{\partial(\Oe \cap \Ot)} N(\widetilde{\B \vp})^2 \, dS(x') \right)^{1/2}
\biggl( \int_{\Oe \setminus \Ot} |\vs|^2 \, dx \biggr)^{1/2}\\
& = o(d).
\end{aligned}
\]
Analogously, one can show that the corresponding expressions involving~$\B\vp$ on~$\Ot \setminus \Oe$ are also of order~$o(d)$. 
\end{proof}

\noindent
To pass from~$\lm^{-1} - \um^{-1}$
to~$\Lum[k] - \Llm$, observe that
\[
\begin{aligned}
\lm^{-1} - \um[k]^{-1} &= \lm^{-2} \biggl( \frac{\lm}{\um[k]} \bigl( \um[k] - \lm \bigr)  \biggr)
= \lm^{-2} \biggl( \um[k] - \lm - \frac{(\um[k] - \lm)^2}{\um[k]} \biggr),
\end{aligned}
\]
where~$(\um[k] - \lm)^2 = O(d^2)$ since~$\Oe$ and~$\Ot$ are at least Lipschitz; see Corollary~6.11 in~\cite{Kozlov2014}.
Note also that if it is the case that~$\Ot \subset \Oe$, we can simplify the previous theorem by removing the second integral in~(\ref{eq:c1_main_eig})
and avoid the use of extensions of eigenfunctions; compare with the statement of Theorem~\ref{t:i:c1} in the introduction.

\section{$C^1$-perturbations of~$C^{1,\alpha}$-domains}
\label{s:app}
Suppose that~$\Oe$ is a $C^{1,\alpha}$-domain and that 
it is possible to characterize the perturbed domain~$\Ot$ by a Lipschitz function~$h$
defined on the boundary~$\partial \Oe$ such that~$(x',x_{\nu}) \in \partial \Ot$ is 
represented by~$x_{\nu} = h(x')$, where~$(x',0) \in \partial \Oe$ and~$x_{\nu}$ is the signed distance to the 
boundary~$\partial \Oe$ (with~$x_{\nu} < 0$ when~$x \in \Oe$).
We assume that~$\nabla h = o(1)$ as~$d \rightarrow 0$ (uniformly).
In this case, we can simplify the expression given in Theorem~\ref{t:c1} and avoid the use of extensions by
stating the formula~(\ref{eq:c1_main}) as a boundary integral.

\begin{theorem}
\label{t:c1_boundary}
Suppose that~$\Oe$ is a~$C^{1,\alpha}$-domain and that~$\Ot$ is  as described above.
Then,
\begin{equation}
\label{eq:holder_main_boundary}
\begin{aligned}
\lm[m]^{-1} - \um[k]^{-1} = {} & 
\tk + o(d)
\end{aligned}
\end{equation}
for~$k = 1,2,\ldots,\Jm$. Here,~$\tau = \tk$ is an eigenvalue of
\begin{equation}
\label{eq:holder_main_boundary_eig}
\tau \ip{\vp}{\vs}{1} =
\lm^{-2} 
\int_{\partial \Oe} h(x') \bigl( 
	(1 - \lm)  \vp \vs
	+ \nabla \vp \cdot \nabla \vs 
	\bigr)  dS(x') \quad \mbox{for all } \vs \in \Xj[m],
\end{equation}
where~$\vp \in \Xj[m]$. Moreover,~$\tk[1], \tk[2], \ldots, \tk[\Jm]$ in~{\rm(}\ref{eq:holder_main_boundary}{\rm)}
run through all eigenvalues of~{\rm(}\ref{eq:holder_main_boundary_eig}{\rm)} counting their multiplicities.
\end{theorem}

\begin{proof}
Since~$\Oe$ is a~$C^{1,\alpha}$-domain, we can use results from the proof of Corollary~6.17 in~\cite{Kozlov2014}.
In that proof, we showed that~$\vp \in C^{1,\alpha}(\Oe)$ and also that~$\vp$ can be extended to a function~$\widetilde{\vp} \in C^{1,\alpha}(\R^n)$
such that
\[
\int_{\Oe \setminus \Ot} \bigl( |\vp(x) - \vp(x',0)|^2 + 
	|\nabla \vp(x) - \nabla \vp(x',0)|^2 \bigr) \, dx
\leq
  C d^{1+\alpha} \, \| \vp \|_{L^2(\Oe)}^2,
\]
with the corresponding estimate holding for~$\widetilde{\vp}$ on~$\Ot \setminus \Oe$.
Hence, Theorem~\ref{t:c1} implies that $\lm^{-1} - \um[k]^{-1}$ is given by
\[
\begin{aligned}
&\lm^{-2} \biggl(  
\int_{\partial \Oe \cap \Ot^c}  
	\int_{0}^{h(x')} \bigl( (1 - \lm)  \vp(x',0)\vs(x',0)
	+ \nabla \vp(x',0) \cdot \nabla \vs(x',0) 
	\bigr) \, dx_{\nu} \, dS(x')\\
& \qquad -
\int_{\partial \Oe \cap \Ot} \int_0^{-h(x')} \bigl( 
	(1 - \lm) \widetilde{\vp}(x',0)\widetilde{\vs}(x',0)
	+ \nabla \widetilde{\vp}(x',0) \cdot \nabla \widetilde{\vs}(x',0)
	\bigr)  \, dx_{\nu} \, dS(x') \biggr)\\
& \quad + o(d).
\end{aligned}
\]
The desired conclusion follows from this statement.
\end{proof}

\def\bibname{References}

\end{document}